\documentclass[12pt]{article}

\usepackage{amsmath,amsthm,amsfonts,amssymb,amscd}
\newtheorem {Lemma}{Lemma}[section]
\newtheorem {Theorem}{Theorem}[section]

\newtheorem {Corollary}{Corollary}[section]

\numberwithin{equation}{section}

\allowdisplaybreaks [4]
\usepackage{fullpage}
\usepackage{enumerate}
\usepackage[utf8]{inputenc}
\usepackage{authblk}

\newenvironment{conditions}
  {\begin{enumerate}[{\upshape (i)}]}
  {\end{enumerate}}

\begin{document}

\title{On distance spectral radius of power hypertrees with given number of pendant paths of fixed length}
\author[1]{Yanna Wang$^{1}$\footnote{E-mail: wangyn@gdcp.edu.cn},  Xuli Qi$^{2}$\footnote{E-mail: xuliqi@hebust.edu.cn}\\
$^{1}$Basic Courses Department, Guangdong Communication Polytechnic, \\
Guangzhou 510650, P.R. China\\
$^{2}$Department of Mathematics, Hebei University of Science and Technology, Shijiazhuang 050018, P.R. China}

\date{}

\maketitle

\begin{abstract}
The distance spectral radius of a connected hypergraph is the largest eigenvalue of the distance matrix of the hypergraph. A pendant path of length $\ell$ with $\ell \ge 1$ in a hypergraph $G$ at $v_{\ell+1}$ is a path $P=(v_1,e_1,v_2,\dots, v_{\ell},e_\ell,v_{\ell+1})$ such that the degree of $v_{\ell+1}$ is at least $2$, the degree of $v_1$ is $1$ and  the degree of $v_i$ is $2$  for $i=2,\dots,\ell$, and the degree of $v$ is $1$ for every $v \in e_i\backslash \{v_{i},v_{i+1}\}$ with $i=1,\dots,\ell$.
We determine the unique hypertree  that maximizes (minimizes, respectively) the distance spectral radius over  all $r$-th power hypertrees of $m$ edges with  $k$ pendant paths of length $\ell$,  where $r\geq 3$, $k,\ell\geq 1$  and $ k\ell<m$. 
\end{abstract}

\section{Introduction}

A  hypergraph $G$ consists of a vertex set $V(G)$  and an edge set $E(G)$,  where every edge in $E(G)$ is a subset of $V(G)$ containing at least two vertices,
 see \cite{Ber}.   For an integer $r\geq 2$, we say that  $G$ is $r$-uniform if every edge of $G$ contains exactly $r$ vertices. An ordinary  graph is just a $2$-uniform hypergraph.
 For distinct vertices $v_0,\dots, v_{p}$ and distinct edges $e_1, \dots, e_p$ of $G$,
the alternating sequence %of vertices and edges
$(v_0,e_1,v_1,\dots,v_{p-1},e_p,v_p)$ such that $v_{i-1}, v_i\in e_i$ for $i=1,\dots,p$ and $e_i\cap e_j=\emptyset$ for $i,j=1,\dots,p$ with $\mid i-j \mid>1$
%$j>i+1$
is a  loose path of $G$ from $v_0$ to $v_p$ of length $p$.
If there is a  loose path from $u$ to $v$ for any $u,v\in V(G)$, then we say that $G$ is connected.
Denote by $P_{m,r}$ the $r$-uniform hypergraph  consisting of a  loose path of length  $m$.
For distinct vertices $v_0,\dots, v_{p-1}$ and distinct edges $e_1, \dots, e_p$,
the alternating sequence %of  vertices and  edges
$(v_0,e_1,v_1,\dots,v_{p-1},e_p,v_0)$ such that $v_{i-1},v_i\in e_i$ for $i=1, \dots, p$ with $v_p=v_0$ and $e_i\cap e_j=\emptyset$ for $i,j=1,\dots,p$ with $\mid i-j \mid>1$ and $\{i,j\}\neq \{1,p\}$
is a  loose cycle of $G$ of length $p$. A hypertree is a connected hypergraph with no  loose cycles.

 For $u,v\in V(G)$, if they are contained in an edge of $G$, then we say that they are adjacent, or $v$ is a neighbor of $u$. For $u\in V(G)$, let $N_G(u)$ be the set of neighbors of $u$ in $G$ and $E_G(u)$ be the set of edges containing $u$ in $G$. The degree of a vertex $u$ in $G$, denoted by $\delta_G(u)$, is $|E_G(u)|$.

%Any  hypergraph $G$ corresponds naturally to a graph  $O_G$  with $V(O_G)=V(G)$
%such that for $u,v\in V(O_G)$,  $\{u, v\}$ is an edge of $O_G$ if and only if $u$ and $v$ are in some edge of $G$. Obviously, an edge of $G$ with size $r$ corresponds naturally to a clique (maximal $2$-connected subgraph)  of $O_G$ with size $r$.

%If $G$ is connected, then the distance between vertices  $u$ and $v$ in $G$ (or $O_G$), denoted by  $d_{G}(u,v)$,  is the length of a shortest  loose path connecting them in $G$.
%Let $G$ be a connected hypergraph  on $n$ vertices. The distance matrix of $G$ is defined as $D(G)=(d_{G}(u,v))_{u, v\in V(G)}$.
%The distance spectral radius of $G$, denoted by $\rho(G)$, is the largest eigenvalue of $D(G)$.
%The  eigenvalues of distance matrices of graphs, arisen from a data communication problem studied by Graham and Pollack \cite{GP} in 1971, have been studied extensively,
%and particularly, the distance spectral radius received much attention, see the
 %survey \cite{AH-2}. We mentioned that the distance spectral radius has also been used as a molecular descriptor, see \cite{BCM,GM}.
%Watanabe et al. \cite{WIS} studied spectral properties of the distance matrix of uniform hypertrees, generalizing some results by Graham and Pollak \cite{GP} and Sivasubramanian~\cite{Si}.

For  $r\geq 2$ and a graph $G$ on $n$ vertices, the $r$-th power of $G$ is defined as the $r$-uniform hypergraph on $n+(r-2)|E(G)|$ vertices with vertex set $V(G)\cup (\cup_{e\in E(G)}V_e)$ and edge set $\{e\cup V_e:e\in E(G)\}$, where $|V_e|=r-2$ for $e\in E(G)$, see \cite{HQS}. Obviously, the $2$-th power of $G$ is $G$ itself. A hypergraph is an $r$-th power hypertree if it is the $r$-th power of some tree. A pendant path of length $\ell$ with $\ell \ge 1$ in a hypergraph $G$ at $v_{\ell+1}$ is a path $P=(v_1,e_1,v_2,\dots, v_{\ell},e_\ell,v_{\ell+1})$ such that $\delta_G(v_{\ell+1})\geq 2$, and $\delta_G(v_1)=1$ and  $\delta_G(v_i)=2$  for $i=2,\dots,\ell$, $\delta_G(v)=1$ for $v \in e_i\backslash \{v_{i},v_{i+1}\}$ with $i=1,\dots,\ell$. If  each edge of a pendant path consists of $r$ vertices, then it is said to be $r$-uniform.
An edge $e$ of a hypergraph $G$ is called a pendant edge of $G$ at $v$ if $v\in e$, the degree of all vertices of $e$ except $v$ in $G$ is one, and $\delta_G(v)\geq 2$.
A pendant vertex is a vertex of degree one.

Let $G$ be a connected hypergraph  on $n$ vertices.
For $u,v\in V(G)$,  the distance between $u$ and $v$ in $G$, denoted by  $d_{G}(u,v)$,  is the length of a shortest  loose  path connecting them in $G$. In particular, $d_{G}(u,u)=0$.
The distance matrix of $G$ is defined as $D(G)=(d_{G}(u,v))_{u, v\in V(G)}$.  The largest
eigenvalue of $D(G)$ is called the distance
spectral radius (or $D$-spectral radius) of $G$, denoted by $\rho(G)$.
%As $D(G)$ is irreducible, the well known
%Perron-Frobenius theorem implies that $\rho(G)$ is simple and there is a unique
%unit positive eigenvector  corresponding to $\rho(G)$,
%which we call the distance Perron vector of $G$, denoted by $x(G)$.
In the literature, the  spectral properties of the distance matrices (or $D$-spectral properties) of graphs were first studied by Graham and
Pollak \cite{GP} in 1971, where, among other results,  they established a connection between the number of negative eigenvalues of the distance matrix and a data communication problem.
In 2012, Watanabe et al.~\cite{WIS} studied the $D$-spectral properties of uniform hypertrees, generalizing some results in \cite{GP}, see also the work of Sivasubramanian~\cite{Si} in 2009.
Now,  the  $D$-spectral properties of graphs or hypergraphs
have been studied extensively,
and among these studies, the $D$-spectral radius received much attention, see the
 survey \cite{AH-2}.
We note that   the  $D$-spectral properties of a hypergraph $H$ are actually those of a particular graph $G$, where each edge of size $r$ of $H$ corresponds to a clique  of size $r$ in $G$.
Lin  and Zhou~\cite{LZ1}
studied the $D$-spectral radius of  uniform hypergraphs and particularly, uniform hypertrees.
 Wang and Zhou \cite{WZ3}   determined the unique hypertree with maximum $D$-spectral radius  in the class of  power hypertrees of given matching number.
Liu et al. \cite{LWL} determined the unique hypertree with minimum $D$-spectral radius among all uniform hypertrees of given matching number, and hypertree with minimum $D$-spectral radius among all uniform hypertrees  of given independence number.
Some more results on the $D$-spectral radius of a hypergraph may be found in \cite{LiZ,LZL,WZ2}.
%
%Li and Zhou \cite{LiZ} determined  the unique hypertrees that minimize the distance spectral radius among all uniform hypertrees with fixed size and one of the parameters, such as number of pendant edges, diameter, and maximum degree that is more than half of the size.
%Wang  and Zhou~\cite{WZ5}  determined the unique graph that maximizes the $D$-spectral radius over  all cacti with $n$ vertices and $k$ cycles, and thus proved a conjecture on the $D$-spectral radius of cacti
%in \cite{SMS}. Notably, they proved the result in the context of hypertrees, showing that in some cases the hypergraph notation is  more convenient and effective.  So the research on the distance spectra of hypergraphs still has its place and warrants further exploration.

In \cite{WZ4}, Wang and Zhou determined the unique tree  that maximizes (respectively minimizes) the $D$-spectral radius over  all trees on $n$ vertices with  $k$ pendant paths of length $\ell$, where $k,\ell\geq 1$ and $k\ell<n-1$.
In this article, we extend the results in \cite{WZ4} from trees to power hypertrees. More precisely,
we determine the unique hypertree  that maximizes (minimizes, respectively) the $D$-spectral radius over  all $r$-th power hypertrees of $m$ edges with  $k$ pendant paths of length $\ell$,   where $r\geq 3$, $k,\ell\geq 1$  and $ k\ell<m$, see Theorem \ref{t0} and \ref{t1} below.
To obtain these results, we give some results related to the entries of the distance Perron vector of  hypergraphs with particular structure.

For integers $m$, $r$, $\ell$, $a$ and $b$ with  $1\leq a\leq b$ and $\ell(a+b)\leq m-1$,
denote by $D_{r,\ell}(m,a,b)$  the $r$-uniform hypertree obtained  from the $r$-uniform path
\[
P_{m-\ell(a+b),r}=(v_1, e_1,v_2, \dots,  v_{m-\ell(a+b)}, e_{m-\ell(a+b)}, v_{m-\ell(a+b)+1})
\]
by attaching $a$ and $b$ pendant $r$-uniform paths of length $\ell$ at $v_1$ and $v_{m-\ell(a+b)+1}$, respectively.

\begin{Theorem} \label{t0}
Let $T$ be an $r$-th power hypertree of $m$ edges with  $k$ pendant paths of length $\ell$, where $r\geq 3$. %$\ell\geq 1$ and $1\leq k<\frac{n-1}{\ell(r-1)}$.

\begin{conditions}

\item If $k=2$ and $1\leq \ell\leq m-1$,   then $\rho(T)\leq \rho (P_{m,r})$ with equality if and only if $T\cong P_{m,r}$.

\item If $k\geq 3$ and $1\leq \ell < \frac{m}{k}$,
then $\rho(T)\leq \rho (D_{r,\ell}(m,\lfloor\frac{k}{2}\rfloor, \lceil\frac{k}{2}\rceil))$ with equality if and only if $T\cong D_{r,\ell}(m,\lfloor\frac{k}{2}\rfloor, \lceil\frac{k}{2}\rceil)$.

\item If $k=1$ and $2\leq \ell\leq m-2$,   then $\rho(T)\leq \rho (D_{r,1}(m,1,2))$ with equality if and only if $T\cong D_{r,1}(m,1,2)$.
\end{conditions}
\end{Theorem}

%For integers $m$, $r$, $\ell$ and $k$ with  $r\geq 3$, $\ell \geq 2$  and $1\leq k<\frac{m}{\ell} $,  denote by $S_{r,\ell}(m,k)$  the hypertree consisting of $k$ pendant paths of length $\ell$ and $m-k\ell$ pendant edges
% paths of length one
%at a common vertex.

For integers $m$, $r$, $\ell$ and $k$ with  $r\geq 3$, $k=1$ and $2\leq \ell \leq m-2$, or $k\geq 2$, $\ell \geq 2$ and $k\ell\leq m-1$, denote by $S_{r,\ell}(m,k)$  the hypertree consisting of $k$ pendant paths of length $\ell$ and $m-k\ell$ pendant edges
at a common vertex.

\begin{Theorem} \label{t1}
Let $T$ be a  $r$-th power hypertree of  $m$ edges with  $k$ pendant paths of length $\ell$, where
$r\geq 3$, $k=1$ and $2\leq \ell \leq m-2$, or $k\geq 2$, $\ell \geq 2$ and $k\ell\leq m-1$.
%$\ell\geq 2$ and $1\leq k<\frac{m}{\ell} $.
Then  $\rho(T)\geq \rho (S_{r,\ell}(m,k))$ with equality if and only if $T\cong S_{r,\ell}(m,k)$.
\end{Theorem}

\section{Preliminaries}

For a connected hypergraph $G$, as $D(G)$ is irreducible, the well known
Perron-Frobenius theorem implies that $\rho(G)$ is simple and there is a unique
unit positive eigenvector  corresponding to $\rho(G)$,
which we call the distance Perron vector of $G$, denoted by $x(G)$.
If $x=(x_{v_1},\dots,x_{v_{n}})^{T}\in\mathbb{R}^{n}$ with
$V(G)=\{v_1,\dots,v_n\}$, then
$x^{\top}D(G)x=2\sum_{\{u,v\}\subseteq V(G)}d_G(u,v)x_ux_v$.
Moreover, if $x$ is unit, then Rayleigh's principle states that  $\rho(G)\geq x^{\top}D(G)x$ with equality   when $x$ has at least one nonnegative component
if and only if $x=x(G)$.
If  $x=x(G)$, then  for any $u\in V(G)$,
$\rho(G)x_u=\sum_{v\in V(G)}d_G(u,v)x_v$,
which we call the $D(G)$-eigenequation at $u$.

For  a connected hypergraph $G$ with  $V_1\subseteq V(G)$, let $\sigma_{G}(V_1)$ be the sum of the entries of the distance Perron vector of $G$ corresponding to the vertices in $V_1$. Furthermore, if all  the vertices of $V_1$ induce a connected subhypergraph $H$ of $G$, then we write $\sigma_{G}(H)$ for  $\sigma_{G}(V_1)$.

For $e\in E(G)$, let $G-e$ be the subhypergraph of $G$ obtained by deleting $e$.
%For $V_1\in V(G)$, let $G[V_1]$ be the subhypergraph of $G$ induced by  $V_1$.

\begin{Lemma}  \label{automorphism}
\cite{LZ1}  Let $G$ be a connected hypergraph with $\eta$ being an automorphism of $G$ and $x=x(G)$. Then $\eta(u)=v$ implies that $x_u=x_v$ for $u,v\in V(G)$.
\end{Lemma}

\begin{Lemma}  \label{66}\cite{WZ5}
Let $T$ be a hypertree with two edges, say $e_1$ and $e_2$. Suppose that $u_i,v_i \in e_i$ for each $i=1,2$, and $d_T(u_1,u_2)=d_T(v_1,v_2)+2$. For $i=1,2$, let $T_i$ be the component of  $T-e_i$ containing $u_i$ and $A_i=\{w\in V(T): d_T(w,u_i)=d_T(w,v_i)\}$. Let $x=x(T)$.
Then
\[
\rho(T)(x_{u_1}-x_{u_2})-\rho(T)(x_{v_1}-x_{v_2})=2(\sigma_{T}(T_2)-\sigma_{T}(T_1))+\sigma_{T}(A_2)-\sigma_{T}(A_1).
\]
\end{Lemma}

\section{Some properties of the distance Perron vector}

In this section, we establish some properties of the distance Perron vector.

\begin{Lemma}  \label{88}
Let $T$ be a hypertree, and $e_1$, $e_2$  be the two edges of $T$ satisfying that  $|e_1|=|e_2|=r$ (where $r\ge 3$).  Suppose that $u_i,v_i \in e_i$ for each $i=1,2$,  $d_T(u_1,u_2)=d_T(v_1,v_2)+2$, and the degree of each vertex in  $e_i\setminus \{u_i,v_i\}$  is one.
For $i=1,2$, let $T_i$ be the component of  $T-e_i$ containing $u_i$.
%and $A_i=\{w\in V(T): d_T(w,u_i)=d_T(w,v_i)\}$. Let $x=x(T)$.
For $w_i\in e_i\setminus\{u_i,v_i\}$ with $i=1,2$,

\[
\rho(T)(x_{u_1}-x_{u_2})-(\rho(T)+1)(x_{w_1}-x_{w_2})=\sigma_{T}(T_2)-\sigma_{T}(T_1)
\]
and
\[
(\rho(T)+1)(x_{w_1}-x_{w_2})-\rho(T)(x_{v_1}-x_{v_2})=(r-2)(x_{w_2}-x_{w_1})+\sigma_{T}(T_2)-\sigma_{T}(T_1).
\]
\end{Lemma}

\begin{proof}
From  the $D(T)$-eigenequations at $u_1$ and  $w_1$, we have
\[
\rho(T)(x_{u_1}-x_{w_1})=x_{w_1}-\sigma_{T}(T_1).
\]
From  the $D(T)$-eigenequations at $w_2$ and $u_2$, we have
\[
\rho(T)(x_{w_2}-x_{u_2})=\sigma_{T}(T_2)-x_{w_2}.
\]
So
\[
\rho(T)(x_{u_1}-x_{u_2})-\rho(T)(x_{w_1}-x_{w_2})=x_{w_1}-x_{w_2}+\sigma_{T}(T_2)-\sigma_{T}(T_1),
 \]
from which the first equation follows.

By Lemma \ref{automorphism}, $x_{w}=x_{w_i}$ for $w\in e_i\backslash \{u_i,v_i,w_i\}$ with $i=1,2$ (if any exists).
Obviously, $d_T(w_1,w_2)=d_T(v_1,v_2)+2$, the component of  $T-e_i$ containing $w_i$ has a single vertex $w_i$ and $\{w\in V(T): d_T(w,w_i)=d_T(w,v_i)\}=V(T_i)\cup (e_i\backslash \{u_i,v_i,w_i\})$ for $i=1,2$.
By Lemma \ref{66},
\[
\rho(T)(x_{w_1}-x_{w_2})-\rho(T)(x_{v_1}-x_{v_2})=2(x_{w_2}-x_{w_1})+\sigma_{T}(T_2)-\sigma_{T}(T_1)+(r-3)(x_{w_2}-x_{w_1}),
\]
from which the second equation  follows.
\end{proof}

%\begin{Lemma}  \label{pendant1}
%\cite{WZ4}
% For $\ell \geq1$, let $G$ be a connected graph with two vertex-disjoint pendant paths, say $P'=u_1\cdots u_\ell u_{\ell +1}$ and $P''=v_1\cdots v_\ell v_{\ell+1}$  at $u_{\ell +1}$ and $v_{\ell+1}$, respectively. Let $x=x(G)$.
%\item[(i)]
%$x_{u_i}-x_{v_i}$ and $x_{u_{i+1}}-x_{v_{i+1}}$ have the same sign for $i=1, \dots, \ell$;
%\item[(ii)]
%If $x_{u_{\ell+1}}>x_{w_{\ell+1}}$, then $x_{u_i}-x_{v_i}<x_{u_{i+1}}-x_{v_{i+1}}$ for $i=1, \dots, \ell$.
%\end {Lemma}

\begin{Lemma}  \label{pendant}
 For $r\geq3$  and $\ell \geq 1$, let $T$ be a  hypertree with two vertex-disjoint $r$-uniform pendant paths of the same length, say $P'=(u_1, e'_1,u_2,\dots,  u_{\ell}, e'_{\ell}, u_{\ell+1})$ and $P''=(v_1, e''_1,v_2, \dots,  v_\ell, e''_{\ell}, v_{\ell+1})$  at $u_{\ell+1}$ and $v_{\ell+1}$, respectively. Suppose that $w'_i \in {e'_i}$ and $w''_i \in {e''_i}$ for $i=1,\dots,\ell$.  Let $x=x(T)$.
Then
\begin{conditions}
\item
$x_{u_i}-x_{v_i}$, $x_{w'_i}-x_{w''_i}$  and $x_{u_{i+1}}-x_{v_{i+1}}$ have the same sign for $i=1, \dots, \ell$.
\item
If $x_{u_{\ell+1}}>x_{v_{\ell+1}}$, then $x_{u_i}-x_{v_i}<x_{u_{i+1}}-x_{v_{i+1}}$ and $x_{w'_i}-x_{w''_i}<x_{u_{i+1}}-x_{v_{i+1}}$  for $i=1, \dots, \ell$.
\end{conditions}
\end {Lemma}

\begin{proof}
We prove Item (i)  by induction on $i$.

For $i=1$,
by Lemma \ref{88},
we have
\[
\rho(T)(x_{u_1}-x_{v_1})-(\rho(T)+1)(x_{w'_{1}}-x_{w''_{1}})=x_{v_{1}}-x_{u_1},
\]
i.e.,
\[
(\rho(T)+1)(x_{u_1}-x_{v_1})=(\rho(T)+1)(x_{w'_{1}}-x_{w''_{1}}),
 \]
so $x_{u_1}-x_{v_1}$ and $x_{w'_{1}}-x_{w''_{1}}$ have the same sign.
On the other hand, by Lemma \ref{88} again,  we have
\[
(\rho(T)+1)(x_{w'_{1}}-x_{w''_{1}})-\rho(T)(x_{u_2}-x_{v_2})=(r-2)(x_{w''_{1}}-x_{w'_{1}})+x_{v_{1}}-x_{u_1},
\]
i.e.,
\[
\rho(T)(x_{u_2}-x_{v_2})=(\rho(T)+r-1)(x_{w'_{1}}-x_{w''_{1}})+x_{u_{1}}-x_{v_1},
  \]
so $x_{w'_{1}}-x_{w''_{1}}$  and $x_{u_2}-x_{v_2}$ have the same sign.

Suppose that $i\geq 2$,  $x_{u_j}-x_{v_j}$, $x_{w'_j}-x_{w''_j}$  and $x_{u_{j+1}}-x_{v_{j+1}}$ have the same sign for $j=1, \dots, i-1$.
By Lemma \ref{88},
we have
\begin{align*}
&\quad \rho(T)\left(x_{u_{i}}-x_{v_{i}}\right)-(\rho(T)+1)\left(x_{w'_{i}}-x_{w''_{i}}\right)\\
&= \left(\sum_{j=1}^{i-1}\left(x_{v_{j}}+(r-2)x_{w''_{j}}\right)+x_{v_{i}}\right)-\left(\sum_{j=1}^{i-1}\left(x_{u_{j}}+(r-2)x_{w'_{j}}\right)+x_{u_{i}}\right)\\
&= \left(\sum_{j=1}^{i}x_{v_{j}}-\sum_{j=1}^{i}x_{u_{j}}\right)+(r-2)\left(\sum_{j=1}^{i-1}x_{w''_{j}}-\sum_{j=1}^{i-1}x_{w'_{j}}\right),
\end{align*}
so
\begin{align*}
(\rho(T)+1)\left(x_{w'_{i}}-x_{w''_{i}}\right)&=\rho(T)\left(x_{u_{i}}-x_{v_{i}}\right)
+\sum_{j=1}^{i}(x_{u_{j}}-x_{v_{j}})\\
&\quad +(r-2)\sum_{j=1}^{i-1}(x_{w'_{j}}-x_{w''_{j}}).
\end{align*}
As $x_{u_{i}}-x_{v_{i}}$,   $\sum_{j=1}^{i}(x_{u_{j}}-x_{v_{j}})$ and  $\sum_{j=1}^{i-1}(x_{w'_{j}}-x_{w''_{j}})$ have the same sign,
 $x_{u_i}-x_{v_i}$ and $x_{w'_{i}}-x_{w''_{i}}$ have the same sign.
On the other hand, by Lemma \ref{88} again,  we have
\begin{align*}
&\quad (\rho(T)+1)\left(x_{w'_{i}}-x_{w''_{i}}\right)-\rho(T)\left(x_{u_{i+1}}-x_{v_{i+1}}\right)\\
&= (r-2)(x_{w''_{i}}-x_{w'_{i}})+\left(\sum_{j=1}^{i-1}\left(x_{v_{j}}+(r-2)x_{w''_{j}}\right)+x_{v_{i}}\right)-\left(\sum_{j=1}^{i-1}\left(x_{u_{j}}+(r-2)x_{w'_{j}}\right)+x_{u_{i}}\right)\\
%\sum_{j=1}^{i}\left(x_{v_{j}}+(r-2)x_{w''_{j}}\right)-\sum_{j=1}^{i}\left(x_{u_{j}}+(r-2)x_{w'_{j}}\right)\\
&= (r-2)(x_{w''_{i}}-x_{w'_{i}})+\left(\sum_{j=1}^{i}x_{v_{j}}-\sum_{j=1}^{i}x_{u_{j}}\right)+(r-2)\left(\sum_{j=1}^{i-1}x_{w''_{j}}-\sum_{j=1}^{i-1}x_{w'_{j}}\right),
\end{align*}
so \begin{align*}
\rho(T)\left(x_{u_{i+1}}-x_{v_{i+1}}\right)&=(\rho(T)+r-1)(x_{w'_{i}}
-x_{w''_{i}})+\sum_{j=1}^{i}(x_{u_{j}}-x_{v_{j}})\\
&\quad +(r-2)\sum_{j=1}^{i-1}(x_{w'_{j}}-x_{w''_{j}}).
\end{align*}
As above, $x_{w'_{i}}-x_{w''_{i}}$ and $x_{u_{i+1}}-x_{v_{i+1}}$ have the same sign.
Therefore, $x_{u_i}-x_{v_i}$, $x_{w'_i}-x_{w''_i}$  and $x_{u_{i+1}}-x_{v_{i+1}}$ have the same sign for $i=1, \dots, \ell$, Item (i) follows.

Suppose that $x_{u_{\ell+1}}>x_{v_{\ell+1}}$. From Item (i), we have $x_{u_{i}}>x_{v_{i}}$ and $x_{w'_{i}}>x_{w''_{i}}$ for $i=1, \dots, \ell$.
For $i=1, \dots, \ell$,
from the $D(T)$-eigenequations at $u_i$, $v_{i}$, $w'_{i}$, $w''_{i}$, $u_{i+1}$ and $v_{i+1}$,
we have
\begin{align*}
&\quad \rho(T)\left(x_{u_{i}}-x_{v_{i}}\right)-(\rho(T)+1)\left(x_{w'_{i}}-x_{w''_{i}}\right)\\
&=\left(\sum_{j=1}^{i-1}\left(x_{v_{j}}+(r-2)x_{w''_{j}}\right)+x_{v_{i}}\right)-\left(\sum_{j=1}^{i-1}\left(x_{u_{j}}+(r-2)x_{w'_{j}}\right)+x_{u_{i}}\right)\\
%\sum_{j=1}^{i}\left(x_{v_{j}}+(r-2)x_{w''_{j}}\right)-\sum_{j=1}^{i}\left(x_{u_{j}}+(r-2)x_{w'_{j}}\right)\\
&= \sum_{j=1}^{i}\left(x_{v_{j}}-x_{u_{j}}\right)+(r-2)\sum_{j=1}^{i-1}(x_{w''_{j}}-x_{w'_{j}})\\
&< 0,
\end{align*}
and
\begin{align*}
&\quad (\rho(T)+1)\left(x_{w'_{i}}-x_{w''_{i}}\right)-\rho(T)\left(x_{u_{i+1}}-x_{v_{i+1}}\right)\\
&= (r-2)(x_{w''_{i}}-x_{w'_{i}})+\left(\sum_{j=1}^{i-1}\left(x_{v_{j}}+(r-2)x_{w''_{j}}\right)+x_{v_{i}}\right)-\left(\sum_{j=1}^{i-1}\left(x_{u_{j}}+(r-2)x_{w'_{j}}\right)+x_{u_{i}}\right)\\
%+\sum_{j=1}^{i}\left(x_{v_{j}}+(r-2)x_{w''_{j}}\right)-\sum_{j=1}^{i}\left(x_{u_{j}}+(r-2)x_{w'_{j}}\right)\\
&= (r-2)(x_{w''_{i}}-x_{w'_{i}})+\sum_{j=1}^{i}\left(x_{v_{j}}-x_{u_{j}}\right)+(r-2)\sum_{j=1}^{i-1}(x_{w''_{j}}-x_{w'_{j}})\\
&< 0,
\end{align*}
so $\rho(T)\left(x_{u_{i}}-x_{v_{i}}\right)<(\rho(T)+1)\left(x_{w'_{i}}-x_{w''_{i}}\right)<\rho(T)\left(x_{u_{i+1}}-x_{v_{i+1}}\right)$. This proves Item (ii).
\end{proof}

\begin{Lemma}  \label{77}
Let $T$ be the hypertree obtained from vertex disjoint
%nontrivial
 hypertrees $T_1$ and $T_2$ with $u\in V(T_1)$ and $v\in V(T_2)$ by adding an $r$-uniform   path $P_{t-1,r}=(v_1, e_1,v_2, \dots,  v_{t-1}, e_{t-1}, v_{t})$ with $v_1=u$ and $v_{t}=v$, where  $t\ge 2$, $r\ge 3$, $V(T_1)\cap V(P_{t-1,r})=\{v_1\}$ and $V(T_2)\cap V(P_{t-1,r})=\{v_t\}$. Suppose that $w_i \in {e_i}$  for $i=1,\dots,t-1$.  Let $x=x(T)$.

\begin{conditions}
\item
If $\sigma_{T}(T_1)-x_u=\sigma_{T}(T_2)-x_v$, then $x_{v_i}=x_{v_{t+1-i}}$ for $i=1,\dots, \left\lfloor\frac{t}{2}\right\rfloor$, and $x_{w_i}=x_{w_{t-i}}$ for $i=1,\dots, \left\lfloor\frac{t-1}{2}\right\rfloor$ if $t\geq 3$.
\item
If $\sigma_{T}(T_1)-x_u<\sigma_{T}(T_2)-x_v$, then
$x_{v_i}>x_{v_{t+1-i}}$ for $i=1,\dots, \left\lfloor\frac{t}{2}\right\rfloor$,   $x_{w_i}>x_{w_{t-i}}$ for $i=1,\dots, \left\lfloor\frac{t-1}{2}\right\rfloor$ if $t\geq 3$,
and
$x_{v_i}-x_{v_{t+1-i}}>x_{v_{i+1}}-x_{v_{t+1-(i+1)}}$ for $i=1,\dots, \left\lfloor\frac{t}{2}\right\rfloor-1$, $x_{v_i}-x_{v_{t+1-i}}>x_{w_{i}}-x_{w_{t-i}}$ for $i=1,\dots, \left\lfloor\frac{t-1}{2}\right\rfloor$.
\end{conditions}
\end {Lemma}

\begin{proof}
The case when $t=2$ in Item (i) follows obviously from the  $D(T)$-eigenequations at $v_1$ and $v_{2}$.

Suppose that $t\geq 3$.

\noindent
{\bf Case 1.} $t$ is even.

Let $p= \frac{t}{2}$. Let $S_1=\sigma_{T}(T_1)+\sum_{j=2}^{p}(x_{v_{j}}+(r-2)x_{w_{j-1}})$ and $S_2=\sigma_{T}(T_2)+\sum_{j=p+1}^{t-1}(x_{v_{j}}+(r-2)x_{w_{j}})$.
We claim that $x_{v_{p-i}}-x_{v_{p+1+i}}$ and $S_2-S_1$ have the same sign for $0\leq i\leq p-1$, $x_{w_{p-1-i}}-x_{w_{p+1+i}}$  and $S_2-S_1$ have the same sign for $0\leq i\leq p-2$. We prove this by induction on $i$.

For $i=0$, from the $D(T)$-eigenequations at $v_p$ and $v_{p+1}$, we have
$\rho(T)(x_{v_p}-x_{v_{p+1}})=S_2-S_1$, so $x_{v_p}-x_{v_{p+1}}$ and $S_2-S_1$ have the same sign. On the other hand, by Lemma \ref{88},
\[
(\rho(T)+1)(x_{w_{p-1}}-x_{w_{p+1}})-\rho(T)(x_{v_p}-x_{v_{p+1}})=S_2-S_1,
\]
so $(\rho(T)+1)(x_{w_{p-1}}-x_{w_{p+1}})=\rho(T)(x_{v_p}-x_{v_{p+1}})+(S_2-S_1)$, and thus $x_{w_{p-1}}-x_{w_{p+1}}$ and $S_2-S_1$ have the same sign.

%Suppose that $1\leq i \leq p-1$, and $x_{v_{p-j}}-x_{v_{p+1+j}}$ and $S_2-S_1$ have the same sign for $0\leq j\leq i-1$, and suppose that $1\leq i \leq p-2$, and $x_{w_{p-1-j}}-x_{w_{p+1+j}}$ and $S_2-S_1$ have the same sign for $0\leq j\leq i-1$.
Suppose that $1\leq i \leq p-2$,  $x_{v_{p-j}}-x_{v_{p+1+j}}$, $x_{w_{p-1-j}}-x_{w_{p+1+j}}$  and $S_2-S_1$ have the same sign for $0\leq j\leq i-1$.
By Lemma \ref{88}, we have
\begin{align*}
&\quad \rho(T)\left(x_{v_{p-i}}-x_{v_{p+1+i}}\right)-(\rho(T)+1)\left(x_{w_{p-1-(i-1)}}-x_{w_{p+1+(i-1)}}\right)\\
&= \left(\sigma_{T}(T_2)+\sum_{j=p+1+i}^{t-1}(x_{v_{j}}+(r-2)x_{w_{j}})\right)
-\left(\sigma_{T}(T_1)+\sum_{j=2}^{p-i}(x_{v_{j}}+(r-2)x_{w_{j-1}})\right)\\
&= \left(S_2-\sum_{j=0}^{i-1}\left(x_{v_{p+1+j}}+(r-2)x_{w_{p+1+j}}\right)\right)-\left(S_1-\sum_{j=0}^{i-1}(x_{v_{p-j}}+(r-2)x_{w_{p-1-j}})\right),
\end{align*}
so
\begin{align*}
\rho(T)\left(x_{v_{p-i}}-x_{v_{p+1+i}}\right)&= (S_2-S_1)\\
&\quad +\sum_{j=0}^{i-1}\left((x_{v_{p-j}}-x_{v_{p+1+j}})+(r-2)(x_{w_{p-1-j}}-x_{w_{p+1+j}})
\right)\\
&\quad +(\rho(T)+1)\left(x_{w_{p-1-(i-1)}}-x_{w_{p+1+(i-1)}}\right),
\end{align*}
and
thus $x_{v_{p-i}}-x_{v_{p+1+i}}$ and $S_2-S_1$ have the same sign. On the other hand, by Lemma \ref{88} again,
\begin{align*}
&\quad (\rho(T)+1)\left(x_{w_{p-1-i}}-x_{w_{p+1+i}}\right)-\rho(T)\left(x_{v_{p-i}}-x_{v_{p+1+i}}\right)\\
&= (r-2)(x_{w_{p+1+i}}-x_{w_{p-1-i}})+\left(\sigma_{T}(T_2)+\sum_{j=p+1+i}^{t-1}(x_{v_{j}}+(r-2)x_{w_{j}})\right)\\
&\quad -\left(\sigma_{T}(T_1)+\sum_{j=2}^{p-i}(x_{v_{j}}+(r-2)x_{w_{j-1}})\right)\\
&= (r-2)(x_{w_{p+1+i}}-x_{w_{p-1-i}})+\left(S_2-\sum_{j=0}^{i-1}\left(x_{v_{p+1+j}}+(r-2)x_{w_{p+1+j}}\right)\right)\\
&\quad -\left(S_1-\sum_{j=0}^{i-1}(x_{v_{p-j}}+(r-2)x_{w_{p-1-j}})\right),
\end{align*}
so
\begin{align*}
&\quad (\rho(T)+r-1)\left(x_{w_{p-1-i}}-x_{w_{p+1+i}}\right)\\
&= (S_2-S_1)+\sum_{j=0}^{i-1}\left((x_{v_{p-j}}-x_{v_{p+1+j}})+(r-2)(x_{w_{p-1-j}}-x_{w_{p+1+j}})
\right)\\
&\quad +\rho(T)\left(x_{v_{p-i}}-x_{v_{p+1+i}}\right),
\end{align*}
and
thus $x_{w_{p-1-i}}-x_{w_{p+1+i}}$ and $S_2-S_1$ have the same sign.

%Now we have showed  that $x_{v_{p-i}}-x_{v_{p+1-i}}$ and $S_2-S_1$ have the same sign for $0\leq i\leq p-1$, $x_{w_{p-1-i}}-x_{w_{p+1+i}}$  and $S_2-S_1$ have the same sign for $0\leq i\leq p-2$. as claimed.

Now we have showed  that $x_{v_{p-i}}-x_{v_{p+1-i}}$, $x_{w_{p-1-i}}-x_{w_{p+1+i}}$  and $S_2-S_1$ have the same sign for $0\leq i\leq p-2$. For the remaining case with $i=p-1$,  by Lemma \ref{88} again,
\begin{align*}
&\quad \rho(T)\left(x_{v_{1}}-x_{v_{2p}}\right)-(\rho(T)+1)\left(x_{w_{1}}-x_{w_{2p-1}}\right)\\
&= \sigma_{T}(T_2)-\sigma_{T}(T_1)\\
&= \left(S_2-\sum_{j=0}^{p-2}\left(x_{v_{p+1+j}}+(r-2)x_{w_{p+1+j}}\right)\right)-\left(S_1-\sum_{j=0}^{p-2}(x_{v_{p-j}}+(r-2)x_{w_{p-1-j}})\right),
\end{align*}
so
\begin{align*}
\rho(T)\left(x_{v_{p-i}}-x_{v_{p+1+i}}\right)&= (S_2-S_1)\\
&\quad +\sum_{j=0}^{p-2}\left((x_{v_{p-j}}-x_{v_{p+1+j}})+(r-2)(x_{w_{p-1-j}}-x_{w_{p+1+j}})
\right)\\
&\quad +(\rho(T)+1)\left(x_{w_{1}}-x_{w_{2p-1}}\right),
\end{align*}
and
thus $x_{v_{p-i}}-x_{v_{p+1+i}}$ and $S_2-S_1$ have the same sign, as claimed.

Suppose first that $\sigma_{T}(T_1)-x_u=\sigma_{T}(T_2)-x_v$.
Note that
\begin{align*}
S_2-S_1&= \sigma_{T}(T_2)-x_v+\sum_{i=0}^{p-1}x_{v_{p+1+i}}+\sum_{i=0}^{p-2}(r-2)x_{w_{p+1+i}}\\
&\quad -\left((\sigma_{T}(T_1)-x_u)+\sum_{i=0}^{p-1}x_{v_{p-i}}+\sum_{i=0}^{p-2}(r-2)x_{w_{p-1-i}}\right)\\
&=  -\sum_{i=0}^{p-1}\left(x_{v_{p-i}}-x_{v_{p+1+i}}\right)-(r-2)\sum_{i=0}^{p-2}\left(x_{w_{p-1-i}}-x_{w_{p+1+i}}\right).
\end{align*}
This requires the common sign in the above claim to be $0$. Thus $x_{v_{p-i}}=x_{v_{p+1-i}}$ for $0\leq i\leq p-1$, i.e., $x_{v_i}=x_{v_{t+1-i}}$ for $i=1,\dots, \lfloor\frac{t}{2}\rfloor$, $x_{w_{p-1-i}}=x_{w_{p+1+i}}$  for $0\leq i\leq p-2$, i.e.,  $x_{w_i}=x_{w_{t-i}}$ for $i=1,\dots, \left\lfloor\frac{t-1}{2}\right\rfloor$ if $t\geq 3$. This proves Item (i).

Suppose next that $\sigma_{T}(T_1)-x_u<\sigma_{T}(T_2)-x_v$.

Note that
\begin{align*}
S_2-S_1&= \sigma_{T}(T_2)-x_v+\sum_{i=0}^{p-1}x_{v_{p+1+i}}+\sum_{i=0}^{p-2}(r-2)x_{w_{p+1+i}}\\
&\quad -\left((\sigma_{T}(T_1)-x_u)+\sum_{i=0}^{p-1}x_{v_{p-i}}+\sum_{i=0}^{p-2}(r-2)x_{w_{p-1-i}}\right)\\
&> -\sum_{i=0}^{p-1}\left(x_{v_{p-i}}-x_{v_{p+1+i}}\right)-(r-2)\sum_{i=0}^{p-2}\left(x_{w_{p-1-i}}-x_{w_{p+1+i}}\right).
\end{align*}
This requires the common sign in the above claim  to be $+$. Thus $x_{v_{p-i}}>x_{v_{p+1-i}}$ for $0\leq i\leq p-1$, i.e., $x_{v_i}>x_{v_{t+1-i}}$ for $i=1,\dots, \lfloor\frac{t}{2}\rfloor$, $x_{w_{p-1-i}}>x_{w_{p+1+i}}$  for $0\leq i\leq p-2$, i.e.,  $x_{w_i}=x_{w_{t-i}}$ for $i=1,\dots, \left\lfloor\frac{t-1}{2}\right\rfloor$ if $t\geq 3$. This proves the first part of  Item (ii).

\noindent
{\bf Case 2.} $t$ is odd.

Let $p= \frac{t-1}{2}$. Let $S_1=\sigma_{T}(T_1)+(r-2)x_{w_{1}}+\sum_{j=2}^{p}(x_{v_{j}}+(r-2)x_{w_{j}})$ and $S_2=\sigma_{T}(T_2)+(r-2)x_{w_{p+1}}+\sum_{j=p+2}^{t-1}(x_{v_{j}}+(r-2)x_{w_{j}})$.
We claim that $x_{w_{p-i}}-x_{w_{p+1+i}}$, $x_{v_{p-i}}-x_{v_{p+2+i}}$ and $S_2-S_1$ have the same sign for $0\leq i\leq p-1$.
 %$x_{v_{p-i}}-x_{v_{p+2+i}}$ and $S_2-S_1$ have the same sign for $0\leq i\leq p-1$, $x_{w_{p-i}}-x_{w_{p+1+i}}$  and $S_2-S_1$ have the same sign for $0\leq i\leq p-1$.
 We prove this by induction on $i$.

For $i=0$, from the $D(T)$-eigenequations at $w_p$ and $w_{p+1}$, we have
$(\rho(T)+1)(x_{w_p}-x_{w_{p+1}})=S_2-S_1$, so $x_{w_p}-x_{w_{p+1}}$ and $S_2-S_1$ have the same sign. On the other hand, by Lemma \ref{88},
\[
\rho(T)(x_{v_p}-x_{v_{p+2}})-(\rho(T)+1)(x_{w_{p}}-x_{w_{p+1}})=S_2-(r-2)x_{w_{p+1}}-(S_1-(r-2)x_{w_{p}}),
\]
so $\rho(T)(x_{v_p}-x_{v_{p+2}})=(S_2-S_1)+(\rho(T)+r-1)(x_{w_{p}}-x_{w_{p+1}})$, and thus $x_{v_{p}}-x_{v_{p+2}}$ and $S_2-S_1$ have the same sign.

%Suppose that $1\leq i \leq p-1$, and $x_{v_{p-j}}-x_{v_{p+2+j}}$ and $S_2-S_1$ have the same sign for $0\leq j\leq i-1$, and suppose that $1\leq i \leq p-1$, and $x_{w_{p-j}}-x_{w_{p+1+j}}$ and $S_2-S_1$ have the same sign for $0\leq j\leq i-1$.
Suppose that $1\leq i \leq p-1$, $x_{w_{p-j}}-x_{w_{p+1+j}}$, $x_{v_{p-j}}-x_{v_{p+2+j}}$ and $S_2-S_1$ have the same sign for $0\leq j\leq i-1$.
By Lemma \ref{88}, we have
\begin{align*}
&\quad (\rho(T)+1)\left(x_{w_{p-i}}-x_{w_{p+1+i}}\right)-\rho(T)\left(x_{v_{p-(i-1)}}-x_{v_{p+2+(i-1)}}\right)\\
&= (r-2)(x_{w_{p+1+i}}-x_{w_{p-i}})+\left(\sigma_{T}(T_2)+\sum_{j=p+2+i}^{t-1}(x_{v_{j}}+(r-2)x_{w_{j}})\right)\\
&\quad -\left(\sigma_{T}(T_1)+\sum_{j=2}^{p-i}(x_{v_{j}}+(r-2)x_{w_{j-1}})\right)\\
&= (r-2)(x_{w_{p+1+i}}-x_{w_{p-i}})+\left(S_2-\sum_{j=0}^{i-1}x_{v_{p+2+j}}-(r-2)\sum_{j=0}^{i}x_{w_{p+1+j}}\right)\\
&\quad -\left(S_1-\sum_{j=0}^{i-1}x_{v_{p-j}}-(r-2)\sum_{j=0}^{i}x_{w_{p-j}}\right),
\end{align*}
so
\begin{align*}
&\quad (\rho(T)+r-1)\left(x_{w_{p-i}}-x_{w_{p+1+i}}\right)\\
&= (S_2-S_1)+\sum_{j=0}^{i-1}(x_{v_{p-j}}-x_{v_{p+2+j}})+(r-2)\sum_{j=0}^{i-1}(x_{w_{p-j}}-x_{w_{p+1+j}})\\
&\quad +\rho(T)\left(x_{v_{p-(i-1)}}-x_{v_{p+2+(i-1)}}\right),
\end{align*}
and
thus $x_{w_{p-i}}-x_{w_{p+1+i}}$ and $S_2-S_1$ have the same sign.
On the other hand, by Lemma \ref{88} again,
\begin{align*}
&\quad \rho(T)\left(x_{v_{p-i}}-x_{v_{p+2+i}}\right)-(\rho(T)+1)\left(x_{w_{p-i}}-x_{w_{p+1+i}}\right)\\
&= \left(\sigma_{T}(T_2)+\sum_{j=p+2+i}^{t-1}(x_{v_{j}}+(r-2)x_{w_{j}})\right)
-\left(\sigma_{T}(T_1)+\sum_{j=2}^{p-i}(x_{v_{j}}+(r-2)x_{w_{j-1}})\right)\\
&= \left(S_2-\sum_{j=0}^{i-1}x_{v_{p+2+j}}-(r-2)\sum_{j=0}^{i}x_{w_{p+1+j}}\right)-\left(S_1-\sum_{j=0}^{i-1}x_{v_{p-j}}-(r-2)\sum_{j=0}^{i}x_{w_{p-j}}\right),
\end{align*}
so
\begin{align*}
\rho(T)\left(x_{v_{p-i}}-x_{v_{p+2+i}}\right)&= (S_2-S_1)\\
&\quad +\sum_{j=0}^{i-1}(x_{v_{p-j}}-x_{v_{p+2+j}})+(r-2)\sum_{j=0}^{i}(x_{w_{p-j}}-x_{w_{p+1+j}})\\
&\quad +(\rho(T)+1)\left(x_{w_{p-i}}-x_{w_{p+1+i}}\right),
\end{align*}
and
thus $x_{v_{p-i}}-x_{v_{p+1+i}}$ and $S_2-S_1$ have the same sign.

Now we have showed  that $x_{w_{p-i}}-x_{w_{p+1+i}}$, $x_{v_{p-i}}-x_{v_{p+2+i}}$ and $S_2-S_1$ have the same sign for $0\leq i\leq p-1$,
% $x_{v_{p-i}}-x_{v_{p+2-i}}$ and $S_2-S_1$ have the same sign for $0\leq i\leq p-1$, $x_{w_{p-i}}-x_{w_{p+1+i}}$  and $S_2-S_1$ have the same sign for $0\leq i\leq p-1$.
as claimed.

Suppose first that $\sigma_{T}(T_1)-x_u=\sigma_{T}(T_2)-x_v$.
Note that
\begin{align*}
S_2-S_1&= \sigma_{T}(T_2)-x_v+\sum_{i=0}^{p-1}x_{v_{p+2+i}}+\sum_{i=0}^{p-1}(r-2)x_{w_{p+1+i}}\\
&\quad -\left((\sigma_{T}(T_1)-x_u)+\sum_{i=0}^{p-1}x_{v_{p-i}}+\sum_{i=0}^{p-1}(r-2)x_{w_{p-i}}\right)\\
&=  -\sum_{i=0}^{p-1}\left(x_{v_{p-i}}-x_{v_{p+2+i}}\right)-(r-2)\sum_{i=0}^{p-1}\left(x_{w_{p-i}}-x_{w_{p+1+i}}\right).
\end{align*}
This requires the above common sign in the above claim to be $0$. Thus $x_{v_{p-i}}=x_{v_{p+2-i}}$ for $0\leq i\leq p-1$, i.e., $x_{v_i}=x_{v_{t+1-i}}$ for $i=1,\dots, \lfloor\frac{t}{2}\rfloor$, $x_{w_{p-i}}=x_{w_{p+1+i}}$  for $0\leq i\leq p-2$, i.e.,  $x_{w_i}=x_{w_{t-i}}$ for $i=1,\dots, \left\lfloor\frac{t-1}{2}\right\rfloor$ if $t\geq 3$. This proves Item (i).

Suppose next that $\sigma_{T}(T_1)-x_u<\sigma_{T}(T_2)-x_v$.

Note that
\begin{align*}
S_2-S_1&= \sigma_{T}(T_2)-x_v+\sum_{i=0}^{p-1}x_{v_{p+2+i}}+\sum_{i=0}^{p-1}(r-2)x_{w_{p+1+i}}\\
&\quad -\left((\sigma_{T}(T_1)-x_u)+\sum_{i=0}^{p-1}x_{v_{p-i}}+\sum_{i=0}^{p-1}(r-2)x_{w_{p-i}}\right)\\
&> -\sum_{i=0}^{p-1}\left(x_{v_{p-i}}-x_{v_{p+2+i}}\right)-(r-2)\sum_{i=0}^{p-1}\left(x_{w_{p-i}}-x_{w_{p+1+i}}\right).
\end{align*}
This requires the above common sign in the above claim to be $+$. Thus $x_{v_{p-i}}>x_{v_{p+2-i}}$ for $0\leq i\leq p-1$, i.e., $x_{v_i}>x_{v_{t+1-i}}$ for $i=1,\dots, \lfloor\frac{t}{2}\rfloor$, $x_{w_{p-i}}>x_{w_{p+1+i}}$  for $0\leq i\leq p-1$, i.e.,  $x_{w_i}=x_{w_{t-i}}$ for $i=1,\dots, \left\lfloor\frac{t-1}{2}\right\rfloor$ if $t\geq 3$. This proves the first part of  Item (ii).

In the following, we prove  the second part of  Item (ii).
By above proof, we have $S_2-S_1> 0$ in either cases.
For $i=1,\dots, \lfloor\frac{t-1}{2}\rfloor$, by Lemma \ref{88}, we have
\begin{align*}
&\quad \rho(T)\left(x_{v_{i}}-x_{v_{t+1-i}}\right)-(\rho(T)+1)\left(x_{w_{i}}-x_{w_{t-i}}\right)\\
&= \sigma_{T}(T_2)-x_v+\sum_{j=1}^{i}x_{v_{t+1-j}}+\sum_{j=1}^{i-1}(r-2)x_{w_{t-j}}\\
&\quad -\left(\sigma_{T}(T_1)-x_u+\sum_{j=1}^{i}x_{v_{j}}+\sum_{j=1}^{i-1}(r-2)x_{w_{j}}\right)\\
&= (\sigma_{T}(T_2)-x_v-\sigma_{T}(T_1)+x_u)+\sum_{j=1}^{i}(x_{v_{t+1-j}}-x_{v_{j}})+(r-2)\sum_{j=1}^{i-1}(x_{w_{t-j}}-x_{w_{j}})\\
&= (S_2-S_1)+\sum_{j=i+1}^{\lfloor\frac{t}{2}\rfloor}(x_{v_{j}}-x_{v_{t+1-j}})+(r-2)\sum_{j=i}^{\lfloor\frac{t-1}{2}\rfloor}(x_{w_{j}}-x_{w_{t-j}})\\
&> 0,
\end{align*}
and for $i=1,\dots, \lfloor\frac{t}{2}\rfloor-1$, by Lemma \ref{88} again, we have
\begin{align*}
&\quad (\rho(T)+1)\left(x_{w_{i}}-x_{w_{t-i}}\right)-\rho(T)\left(x_{v_{i+1}}-x_{v_{t-i}}\right)\\
&= (r-2)(x_{w_{t-i}}-x_{w_{i}})+\sigma_{T}(T_2)-x_v+\sum_{j=1}^{i}x_{v_{t+1-j}}+\sum_{j=1}^{i-1}(r-2)x_{w_{t-j}}\\
&\quad -\left(\sigma_{T}(T_1)-x_u+\sum_{j=1}^{i}x_{v_{j}}+\sum_{j=1}^{i-1}(r-2)x_{w_{j}}\right)\\
&= (\sigma_{T}(T_2)-x_v-\sigma_{T}(T_1)+x_u)\\
&\quad +\sum_{j=1}^{i}(x_{v_{t+1-j}}-x_{v_{j}})+(r-2)\sum_{j=1}^{i}(x_{w_{t-j}}-x_{w_{j}})\\
&= (S_2-S_1)+\sum_{j=i+1}^{\lfloor\frac{t}{2}\rfloor}(x_{v_{j}}-x_{v_{t+1-j}})+(r-2)\sum_{j=i+1}^{\lfloor\frac{t-1}{2}\rfloor}(x_{w_{j}}-x_{w_{t-j}})\\
&> 0,
\end{align*}
so $\rho(T)\left(x_{v_{i}}-x_{v_{t+1-i}}\right)>(\rho(T)+1)\left(x_{w_{i}}-x_{w_{t-i}}\right)>\rho(T)\left(x_{v_{i+1}}-x_{v_{t-i}}\right)$, as desired.
%This proves the second part of  Item (ii).
\end{proof}

 %For a subset $E_1$ of edges of a hyergraph $G$,   $G-E_1$ denotes the subhypergraph obtained from $G$ by deleting all edges in $E_1$.
%For a subset $E_2$ of edges of the complement  $\overline{G}$ of $G$, $G+E_2$ denotes the hypergraph obtained from $G$ by adding all edges in $E_2$, and in particular, if $L_2=\{e\}$, then we write $G+e$ for $G+\{e\}$.

%
%
%\section{Graft transformations that increase or decrease the distance spectral radius}
%In this section, we propose a graft transformations that increase the distance spectral radius under less restrictions.

%First we give some properties related to the entries of the Perron vector of $D_r(n,a,b)$, which will be used in subsequent proof.

%For integers $n$, $r$, $a$ and $b$ with  $1\leq a\leq b$ and $(a+b)r\le n-2$,
%denote by $D_r(n,a,b)$  the tree obtained  from the path $P_{n-r(a+b)}=v_{1}v_{2}\cdots v_{n-r(a+b)}$ by attaching $a$ and $b$ pendant paths of length $r$ at $v_1$ and $v_{n-r(a+b)}$, respectively.

\begin{Lemma}  \label{ab}
 Let $T=D_{r,\ell}(m,a,b)$, where $a\geq 1$, $b\geq a+2$ and $\ell(a+b)\leq  m-1$. Let $t=m-\ell(a+b)+1$.
 Let $x=x(T)$.  Then

 \begin{conditions}

\item
 $x_{v_{1}}-x_{v_{t}}>0$.

\item
 $x_{v_i}-x_{v_{t+1-i}}>x_{v_{i+1}}-x_{v_{t+1-(i+1)}}$ for $i=1,\dots, \lfloor\frac{t}{2}\rfloor-1$ and $x_{v_i}-x_{v_{t+1-i}}>x_{w_{i}}-x_{v_{t-i}}$ for $i=1,\dots, \left\lfloor\frac{t-1}{2}\right\rfloor$.
\end{conditions}
\end{Lemma}

\begin{proof}
Let $V_1$ be all the vertices of the $a$ pendant paths at $v_1$ except $v_1$ and
$P'=(u_1, e'_1,u_2,\dots,\\u_{\ell}, e'_{\ell}, v_{1})$ be one pendant path at $v_1$.
Let $V_2$ be all the vertices of the $b$ pendant paths at $v_t$ except $v_t$ and $P''=(w_1, e''_1,w_2,\dots,  w_{\ell}, e''_{\ell}, v_{t})$ be one pendant path at $v_t$. Suppose that $w'_i \in {e'_i}$  for $i=1,\dots,\ell$ and $w''_i \in {e''_i}$  for $i=1,\dots,\ell$.
By Lemma \ref{pendant} (i),
$\sigma_{T}(V_1)=a\sum_{i=1}^{\ell}(x_{u_i}+(r-2)x_{w'_i})$ and $\sigma_{T}(V_2)=b\sum_{i=1}^{\ell}(x_{w_i}+(r-2)x_{w'_i})$.

We claim that $\sigma_{T}(V_1)< \sigma_{T}(V_2)$.
%It is obvious if $a=0$.
Suppose that
%$a\geq 1$ and
$\sigma_{T}(V_1)\geq \sigma_{T}(V_2)$. As $a<b$, so $\sum_{i=1}^{\ell}(x_{u_i}+(r-2)x_{w'_i})>\sum_{i=1}^{\ell}(x_{w_i}+(r-2)x_{w'_i})$, and by Lemma \ref{pendant} (i), we have  $x_{v_1}>x_{v_t}$.
However, by Lemma \ref{77}, we have $x_{v_1}\leq x_{v_t}$, a contradiction. Thus $\sigma_{T}(V_1)< \sigma_{T}(V_2)$, as claimed.
%By Lemma \ref{77} again, the result follows.
Now the result follows from Lemma \ref{77}.
\end{proof}

\section{Proof of Theorem \ref{t0}}

%In this section, we determine the unique hypertree  that maximizes (minimizes, respectively) the distance spectral radius over  all $r$-th power hypertrees on $n$ vertices with  $k$ pendant paths of length $\ell$,  where $r\geq 3$, $k,\ell\geq 1$  and $ k\ell(r-1)<n-1$.

Let $G$ be a connected $r$-uniform hypergraph with $|E(G)|\geq 1$  and  $v\in V(G)$. For integers $p,q\geq 1$, let
$G(v,p,q)$ be the hypergraph obtained from $G$ by attaching two  pendant $r$-uniform paths with one of length $p$ and another one with length $q$  to $v$.

\begin{Lemma}  \label{path}
\cite{LZ1} Let $G$ be a  connected $r$-uniform hypergraph with $|E(G)|\geq 1$  and $v\in V(G)$. For integers $p\geq q\geq 1$,
$\rho(G(v,p,q))<\rho(G(v,p+1,q-1))$.
\end{Lemma}

Let $G$ be a  hypergraph with $u,v\in V(G)$ and $e_1,\dots,e_s\in E(G)$ such that $u\notin e_i$ and $v\in e_i$ for $1\leq i\leq s$. Let $e'_i=(e_i\setminus \{v\})\cup \{u\} $ for $1\leq i\leq s$. Suppose that $e'_i\not\in E(G)$ for $1\leq i\leq s$. Let $G'$ be the hypergraph with $V(G')=V(G)$ and $E(G')=(E(G)\setminus \{e_1,\dots,e_s\})\cup \{e'_1,\dots,e'_s\}$. Then we say that $G'$ is obtained from $G$ by moving edges $e_1,\dots,e_s$ from $v$ to $u$.

\begin{Lemma}  \cite{WZ2}  \label{3}
For $t\geq 3$, let $G$ be a hypergraph  consisting of $t$ connected subhypergraphs $G_1, \dots, G_t$  such that $|V(G_i)|\geq 2$ for $1\le i\le t$  and  $V(G_i)\cap V(G_j)=\{u\}$ for $1\le i<j\le t$.
Suppose that
$\emptyset \neq I\subseteq \{3,\dots, t\}$.  Let $v\in V(G_2)\setminus \{u\}$ and $G'$ be the hypergraph obtained from $G$ by moving  all the edges containing $u$ in $G_i$ for all $i\in I$ from $u$ to $v$. If $\sigma_{G}(G_1)\geq\sigma_{G}(G_2)$, then $\rho(G)<\rho(G')$.
\end {Lemma}

%\begin{Theorem} \label{t0}
%Let $T$ be an $r$-th power hypertree of order $n$   with $k$ pendant paths of length $\ell$, where $r\geq 3$. %$\ell\geq 1$ and $1\leq k<\frac{n-1}{\ell(r-1)}$.
%
%
%
%(i) If $k=2$ and $1\leq \ell\leq \frac{n-1}{r-1}-1$,   then $\rho(T)\leq \rho (P_{n,r})$ with equality if and only if $T\cong P_{n,r}$.
%
%(ii) If $k\geq 3$ and $1\leq \ell < \frac{n-1}{k(r-1)}$,
%then $\rho(T)\leq \rho (D_{r,\ell}(n,\lfloor\frac{k}{2}\rfloor, \lceil\frac{k}{2}\rceil))$ with equality if and only if $T\cong D_{r,\ell}(n,\lfloor\frac{k}{2}\rfloor, \lceil\frac{k}{2}\rceil)$.
%
%(iii) If $k=1$ and $2\leq \ell\leq \frac{n-1}{r-1}-2$,   then $\rho(T)\leq \rho (D_{r,1}(n,1,2))$ with equality if and only if $T\cong D_{r,1}(n,1,2)$.
%\end{Theorem}

Let $G$ be a connected hypergraph.
For $u\in V(G)$, the status (or transmission) of $u$ in $G$, denoted by $s_G(u)$, is defined to be the sum of distances from $u$ to all other vertices of $G$, i.e., the row sum of $D(G)$ indexed by vertex $u$, i.e., $s_G(u)=\sum_{v\in V(G)}d_{G}(u,v)$. Let $s(G)=\min\{s_G(u):u\in V(G)\}$. It is known that $\rho(G)\geq s(G)$, see ~\cite[p.~24, Theorem~1.1]{Mi}.

\begin{Lemma}  \label{sum}
 Let $T=D_{r,\ell}(n,a,b)$, where $a\geq 1$, $b\geq a+2$ and $\ell (a+b)\leq m-1$. Let $t=m-\ell(a+b)+1$.
 %Let $x=x(T)$.
 Then
 \[
 \rho(T)>a\ell(r-1)(t-1)+\sum_{i=1}^{\lfloor t/2\rfloor}(t-2i+1)+(r-2)\sum_{i=1}^{\lfloor (t-1)/2\rfloor}(t-2i).
 \]
\end{Lemma}

\begin{proof} We choose in $T$ a pendant path $P':=(u_1, e'_1,u_2,\dots,  u_{\ell}, e'_{\ell}, v_{1})$   at $v_1$, and a pendant path
$P'':=(w_1, e''_1,w_2,\dots,  w_{\ell}, e''_{\ell}, v_{t})$  at $v_t$.
Suppose that $w'_i \in {e'_i}$  for $i=1,\dots,\ell$ and $w''_i \in {e''_i}$  for $i=1,\dots,\ell$.
Suppose that $w_i \in {e_i}$  for $i=1,\dots,t-1$.

Let $T'=D_{r,\ell}(m-\ell(b-a),a,a)$.
As $T'$ is a proper induced subhypergraph of $T$ and the distance between any two vertices in $T'$ remains unchanged, we have
$s(T)>s(T')$.   It suffices to show that \[
s(T')\ge a\ell(r-1)(t-1)+\sum_{i=1}^{\lfloor t/2\rfloor}(t-2i+1)+(r-2)\sum_{i=1}^{\lfloor (t-1)/2\rfloor}(t-2i).
\]

Let $p= \lfloor\frac{t}{2}\rfloor$ and $q= \lceil\frac{t}{2}\rceil$.
It can be verified that $s(T')=s_{T'}(v_{q})$.

Note that
\begin{align*}
&\quad \sum_{i=1}^{\ell}\left(d_{T'}(v_{q},u_{i})+(r-2)d_{T'}(v_{q},w'_{i})\right)
+\sum_{i=1}^{\ell}(d_{T'}(v_{q},w_{i})+(r-2)d_{T'}(v_{q},w''_{i}))\\
&> \sum_{i=1}^{\ell}(r-1)\left(d_{T'}(v_{q},v_{1})+d_{T'}(v_{q},v_{t})\right)\\
&= \ell(r-1)(t-1),
\end{align*}
and
\begin{align*}
&\quad \sum_{i=1}^{t}d_{T'}(v_{q},v_{i})+\sum_{i=1}^{t-1}(r-2)d_{T'}(v_{q},w_{i})\\
&= \sum_{i=0}^{p-1}\left(d_{T'}(v_{q},v_{p-i})+d_{T'}(v_{q},v_{q+1+i})\right)+(r-2)\sum_{i=0}^{p-1}\left(d_{T'}(v_{q},w_{p-i})+d_{T'}(v_{q},w_{q+i})\right)\\
&\quad +(p+1-q)(r-2)d_{T'}(v_{q},w_{q})\\
&\geq \sum_{i=0}^{p-1}(q-p+2i+1)+(r-2)\sum_{i=0}^{p-1}(q-p+2i+1)\\
&= \sum_{i=1}^{p}(t-2i+1)+(r-2)\sum_{i=1}^{p}(t-2i+1)\\
&>\sum_{i=1}^{p}(t-2i+1)+(r-2)\sum_{i=1}^{p}(t-2i).
\end{align*}
Thus $s_{T'}(v_{q})>a\ell(r-1)(t-1)+\sum_{i=1}^{\lfloor\frac{t}{2}\rfloor}(t-2i+1)+(r-2)\sum_{i=1}^{\lfloor (t-1)/2\rfloor}(t-2i)$,  as desired.
\end{proof}

\begin{Lemma}  \label{ab-2}
 Suppose that  $b\geq a+2$ and $\ell (a+b)\leq m-1$.  Then $\rho (D_{r,\ell}(m,a+1,b-1))>\rho (D_{r,\ell}(m,a,b))$.
\end{Lemma}

\begin{proof}
Let $T=D_{r,\ell}(m,a,b)$ and $x=x(T)$. Let $t=m-\ell(a+b)+1$.

Let $V_1$ be all the vertices of the $a$ pendant paths at $v_1$ except $v_1$ in $T$ and choose a pendant path
$P':=(u_1, e'_1,u_2,\dots,  u_{\ell}, e'_{\ell}, v_{1})$ at $v_1$.
Let $V_2$ be all the vertices of the $b$ pendant paths at $v_t$ except $v_t$ and choose a pendant path $P'':=(w_1, e''_1,w_2,\dots,  w_{\ell}, e''_{\ell}, v_{t})$ at $v_t$.
Suppose that $w'_i \in {e'_i}$  for $i=1,\dots,\ell$ and $w''_i \in {e''_i}$  for $i=1,\dots,\ell$.
Suppose that $w_i \in {e_i}$  for $i=1,\dots,t-1$.

Let $T'$ be the hypergraph obtained from $T$ by moving  edge $e''_\ell$  at $v_t$    from $v_t$ to $v_1$. Obviously, $T'\cong D_{r,\ell}(m,a+1,b-1)$.

As we pass from   $T$ to $T'$, the distance between a vertex of $V(P'')\backslash \{v_t\}$   and a vertex of $V_2\backslash (V(P'')\backslash \{v_t\})$  is increased by $t-1$,
the distance between a vertex of $V(P'')\backslash \{v_t\}$  and a vertex of $V_1$  is decreased by $t-1$,
and  the distance between a vertex of $V(P'')\backslash \{v_t\}$   and  $v_{i}$  is decreased by $t-2i+1$ with $ i=1, \dots, t$, the distance between a vertex of $V(P'')\backslash \{v_t\}$   and  $w_{i}$  is decreased by $t-2i$ with $ i=1, \dots, t-1$
and the distance between any other vertex pair remains unchanged.  So
\[
\frac{1}{2}(\rho(T')-\rho(T))\geq\frac{1}{2}x^{\top}(D(T')-D(T))x=\sum_{i=1}^{\ell}(x_{w_i}+(r-2)x_{w''_i})L,
\]
where
\[
L=(t-1)\left((b-1)\sum_{i=1}^{\ell}(x_{w_i}+(r-2)x_{w''_i})
-a\sum_{i=1}^{\ell}(x_{u_i}+(r-2)x_{w'_i})\right)-C,
\]
and
\begin{align*}
C&= \sum_{i=1}^{t}(t-2i+1)x_{v_{i}}+(r-2)\sum_{i=1}^{t-1}(t-2i)x_{w_{i}}\\
&= \sum_{i=1}^{\lfloor t/2\rfloor}(t-2i+1)(x_{v_{i}}-x_{v_{t+1-i}})+(r-2)\sum_{i=1}^{\lfloor (t-1)/2\rfloor}(t-2i)(x_{w_{i}}-x_{w_{t-i}}).
\end{align*}

By Lemma \ref{ab} (i), $x_{v_{1}}-x_{v_{t}}>0$, and by Lemma \ref{pendant}(ii), we have  $x_{u_i}-x_{w_i}<x_{v_1}-x_{v_t}$ for $i=1, \dots, \ell$, and $x_{w'_i}-x_{w''_i}<x_{v_{1}}-x_{v_{t}}$  for $i=1, \dots, \ell$.
By Lemma \ref{ab} (ii), $C< \sum_{i=1}^{\lfloor t/2\rfloor}(t-2i+1)(x_{v_{1}}-x_{v_{t}})+(r-2)\sum_{i=1}^{\lfloor (t-1)/2\rfloor}(t-2i)(x_{v_{1}}-x_{v_{t}})$.
By Lemma \ref{sum}, $\rho(T)>a\ell(r-1)(t-1)+\sum_{i=1}^{\lfloor t/2\rfloor}(t-2i+1)+(r-2)\sum_{i=1}^{\lfloor (t-1)/2\rfloor}(t-2i)$.
So, from the distance  eigenequations of $T$ at $v_{1}$ and $v_{t}$,
we have
\begin{align*}
&\quad \rho(T)(x_{v_{1}}-x_{v_{t}})\\
&=  L+(t-1)\sum_{i=1}^{\ell}(x_{w_i}+(r-2)x_{w''_i})\\
&=  2L+ (t-1)\left(a\sum_{i=1}^{\ell}(x_{u_i}+(r-2)x_{w'_i})-(b-2)\sum_{i=1}^{\ell}(x_{w_i}+(r-2)x_{w''_i})\right)+C\\
&\leq 2L+ a(t-1)\sum_{i=1}^{\ell}((x_{u_i}-x_{w_i})+(r-2)(x_{w'_i}-x_{w''_i}))+C\\
&<  2L+ a\ell(r-1)(t-1)(x_{v_1}-x_{v_t})\\
&\quad +\left(\sum_{i=1}^{\lfloor t/2\rfloor}(t-2i+1)+(r-2)\sum_{i=1}^{\lfloor (t-1)/2\rfloor}(t-2i)\right)(x_{v_{1}}-x_{v_{t}})\\
&=  2L+ \left(a\ell(r-1)(t-1)+\sum_{i=1}^{\lfloor t/2\rfloor}(t-2i+1)+(r-2)\sum_{i=1}^{\lfloor (t-1)/2\rfloor}(t-2i)\right)(x_{v_{1}}-x_{v_{t}})\\
&<2L+ \rho(T)(x_{v_{1}}-x_{v_{t}}),
\end{align*}
so $L>0$ and then $\rho(T')>\rho(T)$.
\end{proof}

The case when $t=2$ and $\ell=1$  in previous lemma was proved in \cite{LZ1}.
The case when $t$ is odd and $\ell=1$  in previous lemma was proved in \cite{WZ3}.

Let $T$ be a hypertree with $u\in V(T)$. A branch of $T$ at $u$ is a maximal subhypertree in which $u$ is a pendant vertex. It is obvious that there are $\delta_T(u)$ branches of $T$ at $u$, each containing a neighbor of $u$.

For a subset $V_1$ of vertices of a hypergraph $G$, $G-V_1$ denotes the subgraph of $G$ obtained by deleting all vertices in $V_1$ (and the  edges containing any vertex of $V_1$) from $G$, and in particular, if $V_1=\{u\}$, then we write $G-u$ for $G-\{u\}$.

\begin{proof}[Proof of Theorem \ref{t0}]
Let $T$ be a  $r$-th power hypertree  of $m$ edges with $k$ pendant paths of length $\ell$ that maximizes the $D$-spectral radius. Let $x=x(T)$.

%Suppose first that $k=1$ and $2\leq \ell\leq \frac{n-1}{(r-1)}-2$.
%Recall that $P_{n,r}$ ($D_{r,1}(n,1,2)$, respectively) is the unique hypertree with maximum (second maximum, respectively) distance spectral radius over all $r$-uniform hypertrees  on $n$ vertices, see \cite{SI}.
%As $P_{n,r}$  is an $r$-th power hypertree   on $n$ vertices with two pendant paths of length $\ell$ for each $\ell=1,\dots \frac{n-1}{(r-1)}-1$
%and  $D_{r,1}(n,1,2)$ is a  $r$-th power hypertree   on $n$ vertices with one pendant path of length $\ell$ for each $\ell=2,\dots, \frac{n-1}{(r-1)}-2$,
%it  follows that  $T\cong D_{r,1}(n,1,2)$. This proves Item (i).

%Suppose next that $k=2$ and $1\leq \ell\leq \frac{n-1}{(r-1)}-1$. As $P_{n,r}$  is the unique  hypertree  with maximum distance spectral radius over all $r$-uniform hypertrees  on $n$ vertices, Item (ii) follows.

Suppose first that $k=2$ and $1\leq \ell\leq m-1$.
Recall that $P_{m,r}$ is the unique hypertree with maximum  distance spectral radius over all $r$-uniform hypertrees  of $m$ edges, see \cite{LZ1}. As $P_{m,r}$  is an $r$-th power hypertree  of $m$ edges with two pendant paths of length $\ell$ for each $\ell=1,\dots, m-1$, Item (i) follows.

Suppose next that $k\geq 3$ and $1\leq \ell< \frac{m}{k}$.
Let $V_1$  be the the set of vertices at  which there is a pendant path of length $\ell$ in $T$. Obviously, $|V_1|\geq 1$.

We claim that $|V_1|=2$.

Firstly, we show that $|V_1|\ge 2$.

If $\ell=1$, as $k<m$, we have $|V_1|\geq 2$.

Suppose that $\ell\geq 2$ and $|V_1|=1$. Then $k$ pendant paths of length $\ell$ of $T$
are $k$ branches of $T$ at a common vertex, say $u$, which we denote by $B_1, \dots, B_k$.
Let $T_u=T-\cup_{i=1}^k(V(B_i)\setminus\{u\})$.
As $k\ell<m$,  $|V(T_u)|>1$ and $T_u$ is not a pendant path of length larger than or equal to $\ell$ at $u$,
as otherwise, $T$ would have $k+1$ pendant paths of length $\ell$, which is impossible.

 Let $V'_1=V(B_1)\setminus \{u\}$.

Suppose first that there is  a pendant path, say $Q_1$, at $u$ in $T_u$. Then the length $s$ of $Q_1$ is less than $\ell$. Let  $V''_1=V(Q_1)\setminus\{u\}$ and  $G=T-V'_1-V''_1$. Then
$T\cong G(u, \ell,s)$. It is easy to see that $G(u, \ell+s,0)$ is an $r$-th power hypertree on $n$ vertices with $k$ pendant paths of length $\ell$. However, by Lemma \ref{path}, we have
\[
\rho(T)=\rho (G(u, \ell,s))< \rho(G(u, \ell+1,s-1))< \dots <\rho(G(u, \ell+s,0)),
\]
which is a contradiction.
So  there is no  pendant path  at $u$ in $T_u$. There must be a vertex different from $u$ with degree at least three in $T_u$. Choose such a vertex $v$ so that the distance to $u$ is the largest one.
Let $B$ be the branch of $T_u$ at $u$ containing $v$.
Note that there are two pendant paths at $v$ in $T_u$. Let $\ell_1$ and $\ell_2$ be their lengths with $\ell_1\ge \ell_2$.
Then $T\cong H(v, \ell_1, \ell_2)$ for a well chosen subtree $H$ of $T$.
By repeating this process if necessary, we may finally obtain a hypertree $T'$, which is obtained from $T$ so that the branch $B$ of $T_u$ at $u$ is changed into a pendant path $Q_2$ at $u$ with length $\ell'=\frac{|V(B)|-1}{r-1}$.
By Lemma \ref{path}, we have $\rho(T)<\rho(H(v, \ell_1+\ell_2, 0))<\dots<\rho(T')$.
Note that in $T'$, $B_1, \dots, B_k$ are still $k$ pendant paths of length $\ell$.
Let $V_2=V(Q_2)\setminus\{u\}$.
Let $H'=T'-V'_1-V_2$. Then $T'\cong H'(u,\ell,\ell')$. Evidently,  $H'(u,\ell+\ell',0)$ is an $r$-th power hypertree  of $m$ edges with $k$ pendant paths of length $\ell$.
By Lemma \ref{path}, we have
$\rho(T')=\rho(H'(u,\ell,\ell'))<\rho(H'(u,\ell+\ell',0))$, so $\rho(T)<\rho(H'(u,\ell+\ell',0))$, a contradiction.
 Thus  $|V_1|\geq 2$.

Secondly, we show that $|V_1|\le 2$.

Suppose that $|V_1|\geq 3$. Assume that $v_1,v_2,v_3\in V_1$ and $d_T(v_1, v_2)\ge d_T(v_1, v_3), d_T(v_2, v_3)$.
Let $w$ be the vertex on the path from $v_1$ to $v_2$ such that $d_T(v_3, w)$ is as small as possible.
Let $F_i$ be the component of $T-w$ containing $v_i$ for $i=1,2$. So $T$ has a neighbor of $w$, say $w_i$ in $F_i$ for $i=1,2$, of degree at least two.
Assume that $\sigma_T(F_1)\geq\sigma_T(F_2)$. Denote by $\widetilde{T}$ the hypertree  obtained from  $T$  by moving all edges at $w$ except the unique edge containing  $w$ and $w_1$ and the unique edge containing  $w$ and $w_2$  from $w$ to $v_2$.
Obviously, $\widetilde{T}$ is  an $r$-th power  hypertree  of $m$ edges with $k$ pendant paths of length $\ell$.
By Lemma \ref{3}, $\rho(T)<\rho(\widetilde{T})$, a  contradiction.  Thus $|V_1|\le 2$.

It thus follows that $|V_1|=2$, as claimed.

Now we show that $T\cong D_{r,\ell}(m,a,b)$ for some $a$ and $b$ with $1\leq a\leq b$ and $a+b=k$. It is obvious if $\ell=1$.

Suppose that  $\ell\geq 2$.
Assume that $V_1=\{z_1,z_2\}$.
Let $R$ be the path connecting $z_1$ and $z_2$.
Suppose that there is a vertex of degree at least three different from $z_1$ and $z_2$, say $x$,  on $R$. Let $X_i$ be the component of $T-x$ containing $z_i$ for $i=1,2$. Let $x_i$ be the neighbor of $x$ in $X_i$ for $i=1,2$, of degree at least two.
Assume that $\sigma_T(X_1)\geq\sigma_T(X_2)$. Denote by $\widehat{T}$ be the tree  obtained from  $T$  by moving all edges at $x$ except the unique edge containing  $x$ and $x_1$ and the unique edge containing  $x$ and $x_2$ from $x$ to $z_2$.
Obviously, $\widehat{T}$ is   an $r$-th power  hypertree   of $m$ edges with $k$ pendant paths of length $\ell$.
By Lemma \ref{3}, $\rho(T)<\rho(\widehat{T})$, a  contradiction.  Thus all internal vertices of $R$ (if any exists)  are of degree at most $2$.

Denote by $a$ ($b$, respectively ) the number of pendant paths of length $\ell$ at $z_1$ ($z_2$, respectively). Obviously, $a+b=k$, $\delta_T(z_1)\geq a+1$ and $\delta_T(z_2)\geq b+1$.
If  $\delta_T(z_1)\geq a+2$,
then there is branch, say $B$, of $T$ at $z_1$ not containing $z_2$, which is  not a pendant path of length $\ell$, but, as above by using  Lemma \ref{path},   $B$ can not be a pendant path and it can not have a vertex different from $z_1$ of degree at least three. This is impossible.
So $\delta_T(z_1)= a+1$. Similarly, $\delta_T(z_2)= b+1$. Therefore,  $T\cong D_{r,\ell}(m,a,b)$ for some $a$ and $b$ with $a,b\ge 1$ and  $a+b=k$. Assume that  $a\leq b$.

%If $k=2$, then $a=b=1$ and thus $T\cong D_{r,\ell}(n,1,1)$.
If  $b\geq a+2$, then by Lemma \ref{ab-2}, we have $\rho (D_{r,\ell}(m,a+1,b-1))>\rho (D_{r,\ell}(m,a,b))$, a contradiction. Therefore $b-a=0,1$. That is,  $T\cong D_{r,\ell}(m,\lfloor\frac{k}{2}\rfloor, \lceil\frac{k}{2}\rceil)$. Item (ii) follows.

Suppose  in the following that $k=1$ and $2\leq \ell\leq m-2$.
As $T\ncong P_{m,r}$, the maximum degree of $T$ is at least $3$. Let $z$ be the vertex of degree at least $3$ at which there is a pendant path of length at least $\ell$ in $T$.
Let $B'_1, \dots, B'_{\delta_T(z)}$ be $\delta_T(z)$ branches of $T$ at $z$, such that $B'_1$ contains the unique pendant path of length $\ell$. Obviously, $E(B'_1)\geq \ell$. We claim that $\delta_T(z)=3$  and $E(B'_2)=E(B'_3)=1$.
Suppose first that $\delta_T(z)\geq 4$.
Let $H''=T-(V(B'_1)\cup V(B'_2)\backslash \{z\})$.  Evidently,  $H''(z,E(B'_1)+E(B'_2),0)$ is an $r$-th power  hypertree of $m$ edges with exactly one pendant paths of length $\ell$.
Repeat the process as above, and by Lemma \ref{path} repeatedly, we may finally have
$\rho(T)<\rho(H''(z,E(B'_1)+E(B'_2),0))$,  a contradiction. Thus $\delta_T(z)=3$.
Suppose next that $T\ncong D_{r,1}(n,1,2)$.
Note that $D_{r,1}(n,1,2)$ is a  $r$-th power hypertree   of $m$ edges with exactly one pendant path of length $\ell$ for each $\ell=2,\dots, m-2$.
Repeat the process as above, and by Lemma \ref{path} repeatedly, we may finally have
$\rho(T)<\rho(D_{r,1}(m,1,2))$,  a contradiction.
It  follows that  $T\cong D_{r,1}(m,1,2)$. This proves Item (iii).
\end{proof}

\begin{Corollary}  \label{pendant1}
 Let $T$ be an $r$-th power hypertree of  $m$ edges   with $k$ pendant edges, where $r\geq 3$ and $k<m$.
 then $\rho(T)\leq \rho (D_{r,1}(m,\lfloor\frac{k}{2}\rfloor, \lceil\frac{k}{2}\rceil))$ with equality if and only if $T\cong D_{r,1}(m,\lfloor\frac{k}{2}\rfloor, \lceil\frac{k}{2}\rceil)$.
\end{Corollary}

\section{Proof of Theorem \ref{t1} }

\begin{Lemma}  \label{complete}
\cite{WZ2} Let $G$ be a hypergraph with connected induced subhypergraphs $G_0, H_1$ and $H_2$ such that there are two adjacent vertices $w_1$ and $w_2$ in $G_0$ with $N_{G_0}(w_1)\setminus \{w_2\}=N_{G_0}(w_2)\setminus \{w_1\}$, $V(H_i)\cap V(G_0)=\{w_i\}$ for $i=1,2$, $V(H_1)\cap V(H_2)=\emptyset$, and
$V(G)=V(G_0)\cup V(H_1)\cup V(H_2)$.
 Suppose  that $|V(H_i)|\geq 2$ for $i=1,2$. Let $G'$ be the hypergraph obtained from  $G$  by moving all edges containing $w_2$ except the edges in $E(G_0)$ from $w_2$ to $w_1$. Then $\rho(G)>\rho(G')$.
\end{Lemma}

%Let $T$ be a  $r$-th power hypertree of  $m$ edges with  $k$ pendant paths of length $\ell$, where
%$r\geq 3$, $k=1$ and $2\leq \ell \leq m-2$, or $k\geq 2$, $\ell \geq 2$ and $k\ell\leq m-1$.
%$\ell\geq 2$ and $1\leq k<\frac{m}{\ell} $.
%Then  $\rho(T)\geq \rho (S_{r,\ell}(m,k))$ with equality if and only if $T\cong S_{r,\ell}(m,k)$.

\begin{proof}[Proof of Theorem \ref{t1}]
Let $T$ be a  $r$-th power hypertree  of $m$ edges with $k$ pendant paths of length $\ell$ that minimizes the $D$-spectral radius.

Let $V_1$  be the the set of vertices at which there is a pendant path of length $\ell$ in $T$. Obviously, $|V_1|\geq 1$.
We claim that $|V_1|= 1$. Suppose that $|V_1|\geq 2$ and $u_1,u_2\in V_1$. There is a loose path from $u_1$ to $u_2$ in $G$.  Let $u_3$ be the first vertex of degree at least two on this path from $u_1$ to $u_2$ except $u_1$. Note that it is possible that $u_3=u_2$. Let $T'$ be the hypertree obtained from $T$ by moving all edges at $u_3$ except the unique edge containing  $u_1$ and $u_3$ from $u_3$ to $u_1$.
Obviously, $T'$ is an  $r$-th power hypertree  of $m$ edges with $k$ pendant paths of length $\ell$. By Lemma \ref{complete}, $\rho(T)>\rho(T')$, a contradiction. Thus
 $|V_1|= 1$, as claimed. Let $V_1=\{w\}$.

 %As $k<\frac{n-1}{\ell(r-1)}$, there is a neighbor of $w$, say $z$, not lying on some pendant path of length $\ell$.
 %Note that there are two neighbors of $w$ not lying on the pendant of length $\ell$ if $k=1$, there is  a neighbor of $w$ not lying on any pendant of length $\ell$ if $k\geq2$.

As we have either $k=1$ and $2\leq \ell \leq m-2$, or $k,\ell \geq 2$ and $k\ell\leq m-1$,
there is at least  a neighbor, say $z$ of $w$ not lying on any of the $k$ pendant paths of length $\ell$.
 Suppose that $\delta_T(z)\geq 2$.
 Let $T''$ be the hypertree obtained from $T$ by moving all edges at $z$ except the unique edge containing $w$ and $z$ from $z$ to $w$.
 Obviously, $T''$ is a  $r$-th power hypertree of $m$ edges with $k$ pendant paths of length $\ell$. By Lemma \ref{complete}, $\rho(T)>\rho(T'')$, a contradiction.
Thus $\delta_T(z)= 1$. It follows that $T\cong S_{r,\ell}(m,k)$.
\end{proof}

\vspace{3mm}

\noindent
{\bf Declarations}

\medskip

\noindent
{\bf Competing interests}

On behalf of all authors, the corresponding author states that there is no conflict of interest.
%The authors have no competing interests to declare that are relevant to the content of this article.

\medskip

\noindent {\bf Acknowledgement}

This work was supported by Hebei Natural Science Foundation A2023208006, Hebei Fund for Introducing Overseas Returnees C20230357 and Foundation of China Scholarship Council 202508130088.

\end{document}